\documentclass[12pt,reqno]{amsart}
\usepackage{amsfonts}
\usepackage{amsmath}
\usepackage{mathrsfs}
\usepackage{amssymb,mathtools}
\usepackage{color,amscd,array,graphicx,mathdots,bm,amsfonts,epsfig,amsmath}
\usepackage{amsthm}
\usepackage{enumerate}
\usepackage{geometry}
\usepackage{caption}
\usepackage{subfigure}
\usepackage[T1]{fontenc}
\usepackage{blindtext}
\usepackage{cite}
\usepackage[toc,page]{appendix}
\usepackage{graphicx,color,float} 

\usepackage[colorlinks=true,citecolor=blue]{hyperref}

\newtheorem{theorem}{Theorem}[section]

\newtheorem{lemma}[theorem]{Lemma}

\newtheorem{proposition}[theorem]{Proposition}
\newtheorem{remark}[theorem]{Remark}
\newtheorem{hypothesis}{Hypothesis}

\def\bz{\mathbf{z}}

\def\bs{\mathbf{s}}

\def\EE{\mathbb E}\def\PP{\mathbb P}
\def\RR{\mathbb R}
\def\PP{\mathbb P}
\def\cK{{\mathcal K}}
\def\cD{{\mathcal D}}
\def\cI{{\mathcal I}}\def\cM{{\mathcal M}}\def\cJ{{\mathcal J}}

\def\cX{{\mathcal X}}\def\cY{{\mathcal Y}}
\def\1{\mathbf 1}

\def\<{\left<}\def\>{\right>}

\def\({\left(}\def\){\right)}

\numberwithin{equation}{section}

\begin{document}

\title{Moment asymptotics for super-Brownian motions}

\author{Yaozhong Hu}
\author{Xiong Wang} 
\author{Panqiu Xia}
\author{Jiayu Zheng}

\thanks{Y. Hu is supported by an NSERC Discovery grant and a centennial fund from University of Alberta. J. Zheng is supported by NSFC grant 11901598. X Wang is supported by a research fund from Johns Hopkins University.}

\address{Department of Mathematical and Statistical Sciences, University of Alberta, Edmonton, AB, T6G 2G1, Canada.}
\email{\href{mailto:yaozhong@ualberta.ca}{yaozhong@ualberta.ca}}

\address{Department of Mathematics, Johns Hopkins University, Baltimore, MD 21218, USA.}
\email{\href{mailto:xiong_wang@jhu.edu}{xiong\_wang@jhu.edu}}

\address{Department of Mathematics and Statistics, Auburn University, Auburn, AL 36849, USA.}
\email{\href{mailto:pqxia@auburn.edu}{pqxia@auburn.edu}}

\address{Faculty of Computational Mathematics and Cybernetics, Shenzhen MSU-BIT University, Shenzhen, Guangdong, 518172, China.}
\email{\href{mailto:jyzheng@smbu.edu.cn}{jyzheng@smbu.edu.cn}}

\date{\today}

\begin{abstract}
In this paper, long time and high order moment asymptotics for super-Brownian motions (sBm's) are studied. By using a moment formula for sBm's (e.g. \cite[Theorem 3.1]{arxiv-21-hu-kouritzin-xia-zheng}), precise upper and lower bounds for all positive integer moments and for all time of sBm's for certain initial conditions are achieved. Then, the moment asymptotics as time goes to infinity or as
the moment order goes to infinity follow immediately. Additionally, as an application of the two-sided moment bounds, the tail probability estimates of sBm's are obtained.
\end{abstract}


\maketitle

\normalsize

\section{Introduction}
Supper-Brownian motions (sBm's),  also called the Dawson-Watanabe superprocesses, are 
a class of measure-valued Markov processes (c.f. \cite[etc]{lnm-93-dawson,ams-00-etheridge,springer-02-perkins}) which play key role in branching processes. Because of its closed connection to positive solutions to a class of nonlinear elliptic partial differential equations (c.f. \cite{ams-02-dykin,ams-04-dykin}), sBm's have attracted much attention in the past a few decades. In the one dimensional situation, it is well-known that, with an initial condition $u_0$ being a deterministic finite measure on $\RR$, an sBm has a density with respect to the Lebesgue measure almost surely. This density is the unique weak (in the probabilistic sense) solution to the following stochastic partial differential equation (SPDE),
\begin{align}\label{equ_sbm}
\frac{\partial}{\partial t}u_t(x)=\frac{1}{2}\Delta u_t(x)+\sqrt{u_t(x)}\dot{W}(t,x),
\end{align} 
where $\dot{W}$ denotes the space-time white noise on $\RR_+\times \RR$, i.e., 
\[
 \EE[\dot{W}(t,x)]=0\quad \hbox{and}
 \quad
 \EE[\dot{W}(t,x)\dot{W}(s,y)]=\delta(t-s)\delta(x-y)\,.
 \]

In the present paper, we explore the moment asymptotics of the one-dimensional sBm as time $t\uparrow \infty$ or moment order $n\uparrow \infty$, under certain initial conditions. The moment asymptotics and related intermittency properties for the solution of SPDEs have been intensively studied under the global Lipschitz assumption on diffusion coefficients (cf. \cite[etc]{jsp-95-bertini-cancrini,mams-94-carmona-molchanov,aihp-15-chen,ap-15-chen-dalang,arxiv-22-chen-guo-song,arxiv-21-hu-wang}). However, as shown in equation \eqref{equ_sbm}, the diffusion coefficient for sBm's is not Lipschitz at $0$. Thus current results are not applicable to sBm's. Indeed, from \cite[Proposition 4.7]{arxiv-21-hu-kouritzin-xia-zheng}, one can easily deduce that the $n$-th moment of an sBm is of at most polynomial growth in $t$. In comparison with, e.g. the parabolic Anderson model (the square root term $\sqrt{u_t(x)}$ in \eqref{equ_sbm} is replaced by ${u_t(x)}$) to which the $n$-th moment of the solution is of the exponential growth  (c.f. \cite[Theorem 1.1]{aihp-15-chen}), the growth of the sBm is much slower. This yields that the intermittency property does not hold for sBm's.

On the other hand, one may regard an sBm satisfying \eqref{equ_sbm}, as a special case of the following equation   with  $\beta=\frac{1}{2}$,
\begin{align}\label{equ_spde}
\frac{\partial}{\partial t}u_t(x)=\frac{1}{2}\Delta u_t(x)+u_t(x)^{\beta}\dot{W}(t,x),\quad \beta\in[0,1].
\end{align}
It is well-known that if $\beta=1$, equation \eqref{equ_spde} is the parabolic Anderson model, whose solution is the exponential of the solution to the Kardar-Parisi-Zhang equation \cite{rel-86-kardar-parisi-zhang} through the Hopf-Cole's transformation (c.f. \cite{cmp-97-bertini-giacomin,am-13-hairer}). A sequence of results related to the moment asymptotics, such as intermittency, high peaks (Anderson's localization), (macroscopic) multifractality, etc., have been fully studied (c.f. \cite[etc]{aihp-15-chen,ap-13-conus-joseph-khoshnevisan,ptrf-12-conus-khoshnevisan}). Instead, if $\beta=0$, equation \eqref{equ_spde} degenerates to a stochastic heat equation with additive noise. Even if the intermittency property fails in this case, the solution to \eqref{equ_spde} with $\beta=0$ is still microscopically multifractal with high peaks (c.f. \cite{ap-17-khoshnevisan-kim-xiao,cmp-18-khoshnevisan-kim-xiao}). It is natural to conjecture that sBm's are also microscopically multifractal with high peaks and it is also natural to guess
that the corresponding parameters shall be bounded between those for the parabolic Anderson model and those for additive noise case. However,   this seems to be a  difficult task and we will
not address it here in this work. 

In the following, we introduce some hypotheses on initial condition $u_0$ that will be used to present the main result of this paper.

\begin{hypothesis}\label{hyp_1}
$u_0$ is a positive function on $\RR$ that is two-sided bounded by positive constants, namely,
\[
K_1\leq u_0(x)\leq K_2,
\]
for all $x\in \RR$ with $K_2\geq K_1>0$.
\end{hypothesis}

\begin{hypothesis}\label{hyp_2}
$u_0$ is a finite measure on $\RR$ such that for any $x\in \RR$,
\begin{align}\label{def_L}
\lim_{t\uparrow \infty}t^{\gamma}\int_{\RR}p_t(x-z)u_0(dz)=L\in (0,\infty),
\end{align}
for  some $\gamma\in (0,1)$, where $p_t(x)=\frac{1}{\sqrt{2\pi t}}e^{-\frac{x^2}{2t}}$ denotes the heat kernel.
\end{hypothesis}

Before presenting our main results, let us make some remarks on both hypotheses. First, in the typical context of superprocess  (c.f. \cite{springer-02-perkins}), an sBm can be constructed as the scaling limit of a sequence of branching Brownian motions, where the limit is a random variable taking values $\cD(\RR_+;\cM_F(\RR))$ in the Skorokhod space of finite measures on $\RR$. This requires 
that the initial condition is also a finite measure. Nevertheless, under Hypothesis \ref{hyp_1}, the lower bound of $u_0$ excludes such a possibility. In this case, \cite[Theorem 1.4]{ptrf-88-konno-shiga} proves the existence and uniqueness of sBm's starting at an infinite measure with polynomial growth at infinity. We also refer readers to \cite[Section 2]{arxiv-21-li-pu} for a heuristic discussion on this problem. On the other hand, suppose $u_0=\delta_z$ with some $z\in \RR$, where $\delta$ denotes the Dirac delta measure (c.f. \cite{ap-15-chen-dalang} for SPDEs with rough initial conditions). Then, $u_0$ satisfies Hypothesis \ref{hyp_2}: for every $x\in \RR$,
\[
\lim_{t\uparrow \infty} \sqrt{t}\int_{\RR}p_t(x-y)\delta_z(dy)=\lim_{t\uparrow \infty} \sqrt{t}p_t(x-z) = \frac{1}{\sqrt{2\pi}}.
\]

Now we state our first result of this paper. 
\begin{theorem}\label{thm_tsbd}
Let $u=\{u_t(x): (t,x)\in \RR_+\times \RR\}$ be the solution to \eqref{equ_sbm}. 
Then, there are  positive constants  $K_*$ and $K^*$   independent of $n$, $t$ and $x$
so that the following statements hold.
\begin{enumerate}
 \item[(i)]   Under Hypothesis \ref{hyp_1} 
\begin{align}\label{inq_tsbd1}
K_*^n (1+n! t^{\frac{1}{2}(n-1)})\leq \EE (u_t(x)^n)\leq (K^*)^n (1+n! t^{\frac{1}{2}(n-1)}),
\end{align}
for all $(t,x)\in \RR_+\times \RR$ and for all positive integer $n$. 
\item[(ii)] 
 Under  Hypothesis \ref{hyp_2}, 
 \begin{align}
  \EE (u_t(x)^n)\leq (K^*)^n n! t^{\frac{n-1}{2}-\gamma},
\end{align}
for all $(t,x)\in [C_x ,\infty)\times \RR$ and 
\begin{align}\label{inq_tsbd2}
  \EE (u_t(x)^n)\ge K_*^n n! t^{\frac{n-1}{2}-\gamma}  
\end{align}
for all $(t,x)\in [nC_x \vee 1, \infty)\times \RR$,  where $C_x > 0$, depending on $x$
and satisfying 
\[
\frac{1}{2}L\leq t^{\gamma}\int_{\RR}p_{t}(x-z)u_0(dz)\leq 2L
\]
for all $t>C_x $, where $L$ comes from \eqref{def_L}. 
\end{enumerate}
\end{theorem}

The basic tool in the proof of Theorem \ref{thm_tsbd} is the moment formula for sBm's (see Theorem \ref{thm_mnt}) derived in \cite{arxiv-21-hu-kouritzin-xia-zheng}. Due to formula \eqref{for_mmt0}, the $n$-th moment of an sBm can be represented as the summation of a finite sequence of integrals. Thus it suffices to obtain some sharp bounds for each summand. Fix positive integers $n>n'$ and $(\alpha,\beta,\tau)\in \cJ_{n,n'}$ as in \eqref{for_mmt0}. One may see that the corresponding summand is a space-time integral of heat kernels. Each variable, e.g. $z_i$, appears at most three times in the integrand, where expressions like $p_{t-s_i}(x-z_i)^2$ are counted twice. To apply the semi-group property of heat kernels to calculate the integral, we need to show that the product of three heat kernels in terms of $z_i$ is bounded by that of two heat kernels. In fact, for any $(\alpha,\beta,\tau)\in\cJ_{n,n'}$ with $n \geq 3$ and $n'>0$, there will be expression $p_{t-s_1}(x-z_1)^2$ appearing in the integrand (see \cite[Section 4]{arxiv-21-hu-kouritzin-xia-zheng} for details). Thus, by using the equality 
\[
p(t, x)^2=(4\pi t)^{-\frac{1}{2}}p(t/2,x)
\]
and the semi-group property of heat kernels, one can exactly calculate the integral in $z_1$, which involves a factor of the form $p_{\frac{1}{2}(t-s_1)+s_1}(x-z)$, if $\beta_0=0$; or  $p_{\frac{1}{2}(t-s_1)+(s_1-s_j)}(x-z_j)$ with some $j\in \{2,\dots, n'\}$, if $\beta_0=1$. In the former case, we find that $p_{t-s_2}(x-z_2)^2$ appears in the integrand. Thus, one can further proceed with the integration in $z_2$. On the other hand, assuming $\beta_0 = 1$, it can be proved that either $p_{t-s_2}(x-z_2)^2$, the same as the case previously discussed, or $p_{t-s_2}(x-z_2)p_{s_1-s_2}(z_1-z_2)$ appears in the original integrand. Especially, the existence of $p_{t-s_2}(x-z_2)p_{s_1-s_2}(z_1-z_2)$ implies that $j=2$ and thus after the integration in $z_1$, there exist the factor $p_{\frac{1}{2}(t-s_1)+(s_1-s_2)}(x-z_2)p_{t-s_2}(x-z_2)$, 
 in the remaining integrand. Thanks to the fact that $\frac{1}{2}(t-s_2)\leq \frac{1}{2}(t-s_1)+(s_1-s_2) \leq t-s_2$, we can formulate the next inequality
\begin{align*}
p_{\frac{1}{2}(t-s_2)}(x-z_2)p_{t-s_2}(x-z_2)  \lesssim p_{\frac{1}{2}(t-s_1)+(s_1-s_2)}(x-z_2)p_{t-s_2}(x-z_2)  \lesssim p_{t-s_2}(x-z_2)^2.
\end{align*}
Throughout  the paper,   $A \lesssim B$ (and $A \gtrsim B$, $A \sim B$) means that there are  universal constants $C_1, C_2\in (0, \infty)$  such that $A\le C_1B$ (and $A\ge C_2B$, $C_1B \leq A \leq C_2 B$). 
Notice that $p_{\frac{1}{2}(t-s_2)}(x-z_2)p_{t-s_2}(x-z_2)=(3\pi(t-s_2))^{-\frac{1}{2}}p_{\frac{2}{3}(t-s_2)}(x-z_2)$. The semi-group property of the heat kernel can be applied again when computing two-sided bounds of the integral in $z_2$. Hence, one could expect a desired two-sided bound for $\EE [u_t(x)^n]$ via typical iteration arguments. The detailed proof is given in Section \ref{sec_prfthm1}.

As a consequence of Theorem \ref{thm_tsbd}, we can write the following two propositions of the large time and the high moment asymptotics for sBm's. The proofs are trivial, and thus skipped for simplification.

\begin{proposition}
Let $u=\{u_t(x): (t,x)\in \RR_+\times \RR\}$ be the solution to \eqref{equ_sbm}. Then, for every positive integer $n$ and $x\in \RR$,  under Hypothesis \ref{hyp_1}, 
\begin{align}
\lim_{t\uparrow\infty} \frac{\log\EE (u_t(x)^n)}{\log t}= \frac{1}{2}(n-1);
\end{align}
and under  Hypothesis \ref{hyp_2}, 
\begin{align}
\lim_{t\uparrow \infty} \frac{\log\EE (u_t(x)^n)}{\log t}=\frac{1}{2}(n-1)-\gamma.
\end{align}
\end{proposition}

\begin{proposition}
Let $u=\{u_t(x): (t,x)\in \RR_+\times \RR\}$ be the solution to \eqref{equ_sbm}. Then, under Hypothesis \ref{hyp_1}, for every $(t,x)\in \RR_+\times\RR$, the following convergence holds,
\begin{align}
\lim_{n\uparrow\infty} \frac{\log\EE (u_t(x)^n)}{n\log n}= 1;
\end{align}
and under  Hypothesis \ref{hyp_2}, every $(t,x)\in [C_x ,\infty)\times\RR$, 
\begin{align}\label{inq_hm2}
\limsup_{n\uparrow\infty} \frac{\log\EE (u_t(x)^n)}{n\log n}\leq 1.
\end{align}
\end{proposition}

Under Hypothesis \ref{hyp_2}, we only get an upper bound for the high moment asymptotics (see \eqref{inq_hm2}). This is because as in Theorem \ref{thm_tsbd}, inequality \eqref{inq_tsbd2} holds only for $t>nC_x $. As $n\uparrow \infty$, $nC_x\uparrow \infty$ as well. Thus it seems not possible to have a lower bound for the high moment asymptotics for fixed $t$ by using Theorem \ref{thm_tsbd}.

Another application of Theorem \ref{thm_tsbd} is to get the following tail estimate of sBm's.
\begin{proposition}\label{prop_tail}
Let $u=\{u_t(x): (t,x)\in \RR_+\times \RR\}$ be the solution to \eqref{equ_sbm}. Fix $(t,x)\in \RR_+\times \RR$. Then, under Hypothesis \ref{hyp_1},
\begin{align}\label{inq_te}
-C_1 t^{-\frac{1}{2}}\leq \liminf_{z\to\infty}\frac{\log\PP(u_t(x)>z)}{z}\leq \limsup_{z\to\infty}\frac{\log\PP(u_t(x)>z)}{z}\leq -C_2 t^{-\frac{1}{2}}.
\end{align}
Instead, assume Hypothesis \ref{hyp_2}, and suppose that $t\geq C_x $ where $C_x $ is the same as in Theorem \ref{thm_tsbd}. Then,
\begin{align}\label{inq_teh2}
\limsup_{z\to\infty}\frac{\log\PP(u_t(x)>z)}{z}\leq -C t^{-\frac{1}{2}}.
\end{align}
Here, $C_1$, $C_2$ and $C$ are positive constants independent of $x$.
\end{proposition}

\begin{proposition}\label{thm_ld}
Let $u=\{u_t(x): (t,x)\in \RR_+\times \RR\}$ be the solution to \eqref{equ_sbm}. Fix $x\in \RR$. Then, under Hypothesis \ref{hyp_1} or \ref{hyp_2},
\begin{align}\label{inq_ld}
-C_1 \leq \liminf_{t\to\infty}t^{\frac{1}{2} - \sigma} \log\PP(u_t(x)>t^{\sigma})\leq t^{\frac{1}{2} - \sigma} \limsup_{t\to\infty}\log\PP(u_t(x)>t^{\sigma})\leq -C_2,
\end{align}
for any $\sigma > \frac{1}{2}$ under Hypothesis \ref{hyp_1}; and for any $\sigma \in (\frac{1}{2},\frac{3}{2})$ under Hypothesis \ref{hyp_2}, where $C_1$ and $C_2$ are positive constants independent of $x$.
\end{proposition}

\section{Proof of Theorem \ref{thm_tsbd}}\label{sec_prfthm1}
This section is devoted to the proof of Theorem \ref{thm_tsbd}. To this end, several lemmas related to the upper and lower bounds are proved in Sections \ref{ss_ub} and \ref{ss_lb} respectively. Then, we complete the proof of Theorem \ref{thm_tsbd} in Section \ref{ss_cpl}

\subsection{The upper bound}\label{ss_ub}
For    fix $(\alpha,\beta,\tau)\in \cJ_{n,n'}$ (see   in Section \ref{ss_mnt}), let 
\begin{align}\label{def_ci}
\cX_{t,x} \coloneqq &\prod_{i=1}^n\Big(\int_{\RR} p_{t}(x-z)u_0(dz)\Big)^{1-\alpha_i} \int_{\RR^{n'}}d\bz_{n'}\prod_{i=1}^{n'}\Big(\int_{\RR}p_{s_i}(z_i-z)u_0(dz)\Big)^{1-\beta_i} \nonumber\\
&\times \prod_{i=1}^{|\alpha|}p(t-s_{\tau(i)},x-z_{\tau(i)})\prod_{i=|\alpha|+1}^{2n'}p(s_{\iota_{\beta}(i-|\alpha|)}-s_{\tau(i)},z_{\iota_{\beta}(i-|\alpha|)}-z_{\tau(i)}),
\end{align}
for all $(t,x) \in \RR_+\times \RR$. In this subsection, we will prove the next lemma 
for  a sharp upper bound for $\cX_{t,x}$.

\begin{lemma}\label{lmm_ub}
Let $\cX_{t,x}$ be given as in \eqref{def_ci} with some $(\alpha,\beta,\tau)\in \cJ_{n,n'}$ with positive integer $n$ and nonnegative integer $n'<n$. Then, under Hypothesis \ref{hyp_1}, we have
\begin{align}\label{ine_prest1}
\cX_{t,x}\leq (2\pi)^{-\frac{n'}{2}}K_2^{n-n'} \prod_{i=1}^{n'}(t-s_i)^{-\frac{1}{2}}
\end{align}
for all $(t,x)\in \RR_+\times \RR$.  On the other hand, under Hypothesis \ref{hyp_2},
\begin{align}\label{ine_prest2}
\cX_{t,x} \leq (2\pi)^{-\frac{n'}{2}}(2L)^{n-n'}  t^{-\gamma(n-n')} \prod_{i=1}^{n'}(t-s_i)^{-\frac{1}{2}},
\end{align}
for all $(t,x)\in [C_x ,\infty)\times \RR$ where $C_x >1$ is the same as in Theorem \ref{thm_tsbd}.
\end{lemma}
\begin{proof}
Let $n$ be any positive integer, and let $n'=0$. Then, $(\alpha,\beta,\tau)=(\mathbf{0}_n,\partial,\partial)$, and thus 
\[
\cX_{t,x}=\Big(\int_{\RR}p_t(x-z)u_0(dz)\Big)^n.
\]
Thus, inequalities \eqref{ine_prest1} and \eqref{ine_prest2} are trivially true under Hypotheses \ref{hyp_1} and \ref{hyp_2}, respectively. Particularly, they hold for $n=1$ and all $n'\in\{0, \dots,n-1\}=\{0\}$. Additionally, if $n=2$ and $n'=1$, we have
\begin{align*}
\cX_{t,x}=&\int_{\RR}d\bz_1\Big(\int_{\RR}p_{s_1}(z_1-z)u_0(dz)\Big)p_{t-s_1}(x-z_1)=\int_{\RR}p_{t}(x-z)u_0(dz).
\end{align*}
For the same reason, we can easily verify this lemma in such a situation ($n=2$ and $n'=1$). This allows us to prove this lemma by mathematical induction in $n$.

Let $n> 2$ and let $n'\in \{1,\dots, n-1\}$.  By definition of $\tau$, for any $\{i_1<i_2\}\subset \{1,\dots,2n'\}$ such that $\tau(i_1)=\tau(i_2)=1$, we know that $i_2\leq |\alpha|$. Therefore, $p(t-s_1, x-z_1)^2$ appears in the integrand of $\cX_{t,x}$. 

{\bf Case 1.} Suppose $\beta_1=0$. Recall that $\beta_{n'}=0$. It follows that $|\beta|\leq n'-2$. On the other hand, we know that $|\alpha|+|\beta|=2n'$ and $|\alpha|\leq n$. Thus, $n'\leq n-2$, and we can write, 
\begin{align*}
\cX_{t,x}=\cX_{t,x}^0\cX_{t,x}^1,
\end{align*}
where
\begin{align*}
\cX_{t,x}^0 \coloneqq \int_{\RR^2}p(t-s_1, x-z_1)^2p_{s_1}(z_1-z)u_0(dz)dz_1,
\end{align*}
and
\begin{align*}
\cX_{t,x}^1 \coloneqq & \prod_{i=1}^n\Big(\int_{\RR} p_{t}(x-z)u_0(dz)\Big)^{1-\alpha_i} \int_{\RR^{n'-1}}dz_2\cdots dz_{n'}\prod_{i=2}^{n'}\Big(\int_{\RR}p_{s_i}(z_i-z)u_0(dz)\Big)^{1-\beta_i}\nonumber\\
&\times \prod_{\substack{1\leq i\leq |\alpha|\\ i\notin \{i_1,i_2\}}}p(t-s_{\tau(i)},x-z_{\tau(i)})\prod_{i=|\alpha|+1}^{2n'}p(s_{\iota_{\beta}(i-|\alpha|)}-s_{\tau(i)},z_{\iota_{\beta}(i-|\alpha|)}-z_{\tau(i)}).
\end{align*}
Note that  
\[
p(t-s_1, x-z_1)^2=(4\pi (t-s_1))^{-\frac{1}{2}}p\Big(\frac{1}{2}(t-s_1),x-z_1\Big).
\]
It follows that
\begin{align*}
\cX_{t,x}^0=&(4\pi (t-s_1))^{-\frac{1}{2}}\int_{\RR^2}p\Big(\frac{1}{2}(t-s_1),x-z_1\Big)p_{s_1}(z_1-z)u_0(dz)dz_1\\
=&(4\pi (t-s_1))^{-\frac{1}{2}}\int_{\RR}p\Big(\frac{1}{2}(t-s_1)+s_1,x-z\Big)u_0(dz).
\end{align*}
Additionally, using the fact that $ \frac{1}{2}(t-s_1)+s_1\leq t-s_1+s_1 = t$, we can show that
\begin{align*}
p\Big(\frac{1}{2}(t-s_1)+s_1,x-z\Big)\leq \sqrt{2}p_t(x-z).
\end{align*}
Therefore, under Hypothesis \ref{hyp_1},
\begin{align}\label{ine_prest11}
\cX_{t,x}^0 \leq \frac{K}{\sqrt{2\pi}}(t-s_1)^{-\frac{1}{2}},
\end{align}
for all $(t,x)\in \RR_+\times \RR$; and under Hypothesis \ref{hyp_2},
\begin{align}\label{ine_prest21}
 \cX_{t,x}^0 \leq \frac{2L}{\sqrt{2\pi}}t^{-\gamma}(t-s_1)^{-\frac{1}{2}},
\end{align}
for all $(t,x)\in [C_x ,\infty)\times \RR$.

Let $\alpha' \coloneqq (\alpha_1,\dots,\alpha_{i_1-1},\alpha_{i_1+1},\dots \alpha_{i_2-1},\alpha_{i_2+1},\dots, \alpha_n)$,  and let $\beta' \coloneqq (\beta_2,\dots, \beta_{n'})$. Then, it can be verified that $[\alpha',\beta']\in\cI_{n-2,n'-1}$ (c.f. Section \ref{ss_mnt}). Let $\tau':\{1,\dots, 2(n'-1)\}\to \{1,\dots, n'-1\}$ be given by
\begin{align*}
\tau'(i) \coloneqq \begin{cases}
\tau(i)-1, & 1\leq i\leq i_1-1,\\
\tau(i+1)-1, & i_1\leq i\leq  i_2-2,\\
\tau(i+2)-1, & i_2-1\leq i\leq 2(n'-1).
\end{cases}
\end{align*}
Then, we can also show that $\tau'\in \cK_{n-2,n'-1}^{\alpha',\beta'}$. Moreover, $\cX_{t,x}^1$ can be represented as   follows,
\begin{align*}
\cX_{t,x}^1=&\prod_{i=1}^{n-2}\Big(\int_{\RR} p_{t}(x-z)u_0(dz)\Big)^{1-\alpha'_i} \int_{\RR^{n'-1}}dz_2\cdots dz_{n'}\prod_{i=1}^{n'-1}\Big(\int_{\RR}p_{s_{i+1}}(z_{i+1}-z)u_0(dz)\Big)^{1-\beta'_i}\\
&\times \prod_{i=1}^{|\alpha'|}p(t-s_{\tau'(i)+1},x-z_{\tau'(i)+1})\prod_{i=|\alpha'|+1}^{2(n'-1)}p(s_{\iota_{\beta'}(i-|\alpha'|)+1}-s_{\tau'(i)+1},z_{\iota_{\beta'}(i-|\alpha'|)+1}-z_{\tau'(i)+1}).
\end{align*}
By using the induction hypothesis, we have
\begin{align}\label{ine_prest1n}
\cX_{t,x}^1\leq (2\pi)^{-\frac{n'-1}{2}}K_2^{n-n'-1} \prod_{i=2}^{n'}(t-s_i)^{-\frac{1}{2}}
\end{align}
for all $(t,x)\in \RR_+\times \RR$, under Hypothesis \ref{hyp_1}; and
\begin{align}\label{ine_prest2n}
 \cX_{t,x}^1\leq (2\pi)^{-\frac{n'-1}{2}}(2L)^{n-n'-1} t^{-\gamma(n-n'-1)} \prod_{i=2}^{n'}(t-s_i)^{-\frac{1}{2}},
\end{align}
for all $(t,x)\in [C_x ,\infty)\times \RR$ under Hypothesis \ref{hyp_2}.

Hence, under Hypothesis \ref{hyp_1}, inequality \eqref{ine_prest1} is a consequence of inequalities \eqref{ine_prest11} and \eqref{ine_prest1n}; and under Hypothesis \ref{hyp_2}, inequality \eqref{ine_prest2} is a consequence of inequalities \eqref{ine_prest21} and \eqref{ine_prest2n}. This completes the proof of this lemma under the assumption that $\beta_1=0$.

{\bf Case 2.} Suppose that $\beta_1=1$, then there exists $j\geq 2$ such that $\tau(|\alpha|+1)=j$. Thus we find the following expression in the integrand of $\cX_{t,x}$,
\[
p(t-s_1,x-z_1)p(t-s_1,x-z_1)p(s_1-s_{j},z_1-z_j).
\]
Integrating in $z_1$ and by a similar argument as in Case 1, we have 
\begin{align}\label{inq_cx1}
 \cX_{t,x}\leq \frac{1}{\sqrt{2\pi}}(t-s_1)^{-\frac{1}{2}}\cX_{t,x}^2,
\end{align}
where
\begin{align}\label{def_cx'}
\cX_{t,x}^2 \coloneqq & \prod_{i=1}^n\Big(\int_{\RR} p_{t}(x-z)u_0(dz)\Big)^{1-\alpha_i} \int_{\RR^{n'-1}}dz_2\cdots dz_{n'}\prod_{i=2}^{n'}\Big(\int_{\RR}p_{s_i}(z_i-z)u_0(dz)\Big)^{1-\beta_i}\nonumber\\
&\times p\Big(\frac{1}{2}(t-s_1)+(s_1-s_j),x-z_j\Big)\prod_{\substack{1\leq i\leq |\alpha|\\ i\notin \{i_1,i_2\}}}p(t-s_{\tau(i)},x-z_{\tau(i)})\nonumber\\
&\times \prod_{i=|\alpha|+2}^{2n'}p(s_{\iota_{\beta}(i-|\alpha|)}-s_{\tau(i)},z_{\iota_{\beta}(i-|\alpha|)}-z_{\tau(i)})\nonumber\\
\leq &\sqrt{2}\prod_{i=1}^n\Big(\int_{\RR} p_{t}(x-z)u_0(dz)\Big)^{1-\alpha_i} \int_{\RR^{n'-1}}dz_2\cdots dz_{n'}\prod_{i=2}^{n'}\Big(\int_{\RR}p_{s_i}(z_i-z)u_0(dz)\Big)^{1-\beta_i}\nonumber\\
&\times p(t-s_j,x-z_j)\prod_{\substack{1\leq i\leq |\alpha|\\ i\notin \{i_1,i_2\}}}p(t-s_{\tau(i)},x-z_{\tau(i)}) \nonumber\\
&\times \prod_{i=|\alpha|+2}^{2n'}p(s_{\iota_{\beta}(i-|\alpha|)}-s_{\tau(i)},z_{\iota_{\beta}(i-|\alpha|)}-z_{\tau(i)}).
\end{align}
Let $\alpha'' \coloneqq (1,\alpha_1,\dots,\alpha_{i_1-1},\alpha_{i_1+1},\dots \alpha_{i_2-1},\alpha_{i_2+1},\dots, \alpha_n)$,  and let $\beta'' \coloneqq (\beta_2,\dots, \beta_{n'})$. Then, $[\alpha'',\beta'']\in\cI_{n-1,n'-1}$. Let $\tau'':\{1,\dots, 2(n'-1)\}\to \{1,\dots, n'-1\}$ be given by
\begin{align*}
\tau''(i) \coloneqq \begin{cases}
j-1, & i=1\\
\tau(i-1)-1, & 2\leq i\leq i_1,\\
\tau(i)-1, & i_1< i< i_2,\\
\tau(i+1)-1, & i_2\leq i\leq |\alpha''|,\\
\tau(i+2)-1, & |\alpha''|< i\leq 2(n'-1).
\end{cases}
\end{align*}
Then, $\tau''\in \cK_{n-1,n'-1}^{\alpha'',\beta''}$, and we can rewrite \eqref{def_cx'} as follows
\begin{align*}
\cX_{t,x}^2 \leq &\sqrt{2}\prod_{i=1}^{n-1}\Big(\int_{\RR} p_{t}(x-z)u_0(dz)\Big)^{1-\alpha''_i} \int_{\RR^{n'-1}}dz_2\cdots dz_{n'}\prod_{i=1}^{n'-1}\Big(\int_{\RR}p_{s_{i+1}}(z_{i+1}-z)u_0(dz)\Big)^{1-\beta''_i}\nonumber\\
&\times \prod_{i=1}^{|\alpha''|}p(t-s_{\tau''(i)+1},x-z_{\tau'(i)+1}) \prod_{i=|\alpha''|+1}^{2(n'-1)}p(s_{\iota_{\beta''}(i-|\alpha''|)+1}-s_{\tau''(i)+1},z_{\iota_{\beta''}(i-|\alpha''|)+1}-z_{\tau''(i)+1}).
\end{align*}
By induction hypothesis again, we have
\begin{align*}
\cX_{t,x}^2 \leq \sqrt{2}(2\pi)^{-\frac{n'-1}{2}}K_2^{n-n'} \prod_{i=2}^{n'}(t-s_i)^{-\frac{1}{2}}
\end{align*}
for all $(t,x)\in \RR_+\times \RR$, under Hypothesis \ref{hyp_1}; and
\begin{align}\label{ine_prest2n1}
 \cX_{t,x}^2 \leq \sqrt{2}(2\pi)^{-\frac{n'-1}{2}}(2L)^{n-n'} t^{-\gamma(n-n'-1)} \prod_{i=2}^{n'}(t-s_i)^{-\frac{1}{2}},
\end{align}
for all $(t,x)\in [C_x ,\infty)\times \RR$ under Hypothesis \ref{hyp_2}.

Thus, under Hypothesis \ref{hyp_1}, inequality \eqref{ine_prest1} is a consequence of inequality \eqref{inq_cx1}, inequality \eqref{ine_prest2} is a consequence of inequalities \eqref{inq_cx1} and \eqref{ine_prest2n1}. The proof of this lemma is thus complete.
\end{proof}

\subsection{The lower bound}\label{ss_lb}
Recalling moment formula \eqref{for_mmt0}, it follows that for all $n\geq 2$,
\begin{align*}
\EE (u_{t}(x)^n)\geq &\Big(\int_{\RR} p_{t}(x-z)u_0(dz)\Big)^n+\sum_{(\alpha,\beta,\tau)\in \cJ_{n,n-1}}\prod_{i=1}^n\Big(\int_{\RR} p_{t}(x-z)u_0(dz)\Big)^{1-\alpha_i}\nonumber\\
&\times \int_{\mathbb{T}_{n-1}^t}d\bs_{n-1}\int_{\RR^{n-1}}d\bz_{n-1}\prod_{i=1}^{n-1}\Big(\int_{\RR}p_{s_i}(z_i-z)u_0(dz)\Big)^{1-\beta_i} \nonumber\\
&\times \prod_{i=1}^{|\alpha|}p(t-s_{\tau(i)},x-z_{\tau(i)})\prod_{i=|\alpha|+1}^{2(n-1)}p(s_{\iota_{\beta}(i-|\alpha|)}-s_{\tau(i)},z_{\iota_{\beta}(i-|\alpha|)}-z_{\tau(i)}).
\end{align*}
In fact, for every $(\alpha,\beta,\tau)\in \cJ_{n,n-1}$, we know that $\alpha = \alpha_* \coloneqq \1_{n}$ and $\beta=\beta_* \coloneqq (\1_{n-2},0)$, where $\1_{n}$ denotes the $n$-dimensional vector with unit coordinates. Thus, we can write
\begin{align*}
\EE (u_{t}(x)^n)\geq& \Big(\int_{\RR} p_{t}(x-z)u_0(dz)\Big)^n\\
&+\sum_{\tau\in \cK_{n,n-1}^{\alpha_*,\beta_*}} \int_{\mathbb{T}_{n-1}^t}d\bs_{n-1}\int_{\RR^{n-1}}d\bz_{n-1}\Big(\int_{\RR}p_{s_{n-1}}(z_{n-1}-z)u_0(dz)\Big) \nonumber\\
&\quad\times\prod_{i=1}^{n}p(t-s_{\tau(i)},x-z_{\tau(i)})\prod_{i=n+1}^{2(n-1)}p(s_{\iota_{\beta}(i-n)}-s_{\tau(i)},z_{\iota_{\beta}(i-n)}-z_{\tau(i)}).
\end{align*}

\noindent Let $(\theta_1,\dots,\theta_n)\in (0,1]^{n}$. Fix $\tau\in \cK_{n,n-1}^{\alpha_*,\beta_*}$, and let
\begin{align}\label{def_ytx}
\cY_{t,x} \coloneqq &\int_{\RR^{n-1}}d\bz_{n-1}\Big(\int_{\RR}p_{s_{n-1}}(z_{n-1}-z)u_0(dz)\Big)\prod_{i=1}^{n}p(\theta_i(t-s_{\tau(i)}),x-z_{\tau(i)})\nonumber\\
&\times\prod_{i=n+1}^{2(n-1)}p(s_{\iota_{\beta}(i-n)}-s_{\tau(i)},z_{\iota_{\beta}(i-n)}-z_{\tau(i)}).
\end{align}

 The next lemma is the main result of this subsection.
\begin{lemma}\label{lmm_lb}
Let $\cY_{t,x}$ be defined as in \eqref{def_ytx}. Then,
\begin{align}\label{inq_lb1}
\cY_{t,x}\geq K_1(4\pi)^{-\frac{n}{2}} \prod_{i=1}^{n-1}(t-s_i)^{-\frac{1}{2}}, 
\end{align}
under Hypothesis \ref{hyp_1} for all $(t,x)\in \RR_+\times \RR$; 
and
\begin{align}\label{inq_lb2}
\cY_{t,x}\geq \frac{1}{2}L (4\pi)^{-\frac{n}{2}}t^{-\gamma}\prod_{i=1}^{n-1}(t-s_i)^{-\frac{1}{2}}
\end{align}
under Hypothesis \ref{hyp_2} for all $(t,x)\in [(\sum_{i=1}^n\theta_i^{-1})C_x ,\infty)\times \RR$, where $C_x $ is a positive constant such that for all $t\geq C_x $,
\[
t^{\gamma}\int_{\RR}p_t(x-y)u_0(dy)\geq \frac{1}{2}L
\]
\end{lemma}
\begin{proof}
We prove this lemma following similar ideas as in Lemma \ref{lmm_ub}. Firstly, suppose $n=2$, we have
\begin{align*}
\cY_{t,x}=&\int_{\RR}d\bz_1\Big(\int_{\RR}p_{s_1}(z_1-z)u_0(dz)\Big)p(\theta_1(t-s_1),x-z_1)\\
=&\int_{\RR}p_{\theta_1(t-s_1)+s_1}(x-z)u_0(dz).
\end{align*}
Then, it is clear that this lemma holds for $n=2$. In the next step, choose $i_1<i_2$ such that $\tau(i_1)=\tau(i_2)=1$ and choose $j$ such that $\tau(n+1)=j$. Then, $1\leq i_1<i_2\leq n$ and $j\geq 2$, and we can write $\cY_{t,x}$ as follows,
\begin{align*}
\cY_{t,x}=&\int_{\RR^{n-2}}dz_2\cdots dz_{n-1}\Big(\int_{\RR}p_{s_{n-1}}(z_{n-1}-z)u_0(dz)\Big)\prod_{\substack{1\leq i\leq n\\ i\notin\{i_1,i_2\}}}p(\theta_i(t-s_{\tau(i)}),x-z_{\tau(i)})\nonumber\\
&\times\prod_{i=n+2}^{2(n-1)}p(s_{\iota_{\beta}(i-n)}-s_{\tau(i)},z_{\iota_{\beta}(i-n)}-z_{\tau(i)})\\
&\times \int_{\RR}dz_1p(\theta_{i_1}(t-s_1),z-z_1)p(\theta_{i_2}(t-s_1),z-z_1)p(s_1-s_j,z_1-z_j).
\end{align*}
Taking account of the fact that
\[
p_s(x)p_t(x)=(2\pi(s+t))^{-\frac{1}{2}}p\Big(\frac{st}{s+t},x\Big),
\]
for all $s,t\in \RR_+$ and $x\in \RR$, we can write
\begin{align*}
p(\theta_{i_1}(t-s_1),z-z_1)p(\theta_{i_2}(t-s_1),z-z_1)=&\frac{p\Big(\frac{\theta_{i_1}\theta_{i_2}}{\theta_{i_1}+\theta_{i_2}}(t-s_1),z-z_1\Big)}{(2\pi (\theta_{i_1}+\theta_{i_2})(t-s_1))^{\frac{1}{2}}}\\
\geq &(4\pi (t-s_1))^{-\frac{1}{2}}p\Big(\frac{\theta_{i_1}\theta_{i_2}}{\theta_{i_1}+\theta_{i_2}}(t-s_1),z-z_1\Big).
\end{align*}
It follows that
\begin{align*}
&\int_{\RR}dz_1p(\theta_{i_1}(t-s_1),z-z_1)p(\theta_{i_2}(t-s_1),z-z_1)p(s_1-s_j,z_1-z_j)\\
\geq&(4\pi (t-s_1))^{-\frac{1}{2}}p\Big(\frac{\theta_{i_1}\theta_{i_2}}{\theta_{i_1}+\theta_{i_2}}(t-s_1)+(s_1-s_j),z-z_j\Big).
\end{align*}
On the other hand, it is clear that
\begin{align*}
\frac{\theta_{i_1}\theta_{i_2}}{\theta_{i_1}+\theta_{i_2}}(t-s_j)\leq \frac{\theta_{i_1}\theta_{i_2}}{\theta_{i_1}+\theta_{i_2}}(t-s_1)+(s_1-s_j)\leq t-s_j
\end{align*}
Thus, let
\[
\theta':=\Big[\frac{\theta_{i_1}\theta_{i_2}}{\theta_{i_1}+\theta_{i_2}}(t-s_1)+(s_1-s_j)\Big]/(t-s_j),
\]
we have $\theta'\in [\frac{\theta_{i_1}\theta_{i_2}}{\theta_{i_1}+\theta_{i_2}},1]\subset (0,1]$. Furthermore, we can write,
\begin{align*}
\cY_{t,x}\geq &(4\pi)^{-\frac{1}{2}}(t-s_1)^{-\frac{1}{2}}\int_{\RR^{n-2}}dz_2\cdots dz_{n-1}\Big(\int_{\RR}p_{s_{n-1}}(z_{n-1}-z)u_0(dz)\Big)p(\theta'(t-s_j),z-z_j)\nonumber\\
&\times\prod_{\substack{1\leq i\leq n\\ i\notin\{i_1,i_2\}}}p(\theta_i(t-s_{\tau(i)}),x-z_{\tau(i)})\prod_{i=n+2}^{2(n-1)}p(s_{\iota_{\beta}(i-n)}-s_{\tau(i)},z_{\iota_{\beta}(i-n)}-z_{\tau(i)}).
\end{align*}
Similarly to   Case 2 of the proof of Lemma \ref{lmm_ub}, we can find $\tau'\in \cK_{n-1,n-2}^{\alpha_0',\beta_0'}$ where $\alpha_0'=\1_{n-1}$ and $\beta_0'=(\1_{n-3},0)$ such that 
\begin{align}\label{ine_cy}
\cY_{t,x}\geq &(4\pi)^{-\frac{1}{2}}(t-s_1)^{-\frac{1}{2}}\cY_{t,x}^0,
\end{align}
with
\begin{align*}
\cY_{t,x}^0=&\int_{\RR^{n-2}}dz_2\cdots dz_{n-1}\Big(\int_{\RR}p_{s_{n-1}}(z_{n-1}-z)u_0(dz)\Big)\prod_{i=1}^{n-1}p(\theta''_i(t-s_{\tau'(i)+1}),x-z_{\tau'(i)+1})\\
&\times\prod_{i=n}^{2(n-2)}p(s_{\iota_{\beta'}(i-n-1)+1}-s_{\tau'(i)+1},z_{\iota_{\beta'}(i-n-1)+1}-z_{\tau'(i)+1}),
\end{align*}
and
\begin{align*}
\theta''_i=\begin{cases}
\theta', & i=1\\
\theta_{i-1}, & 2\leq i\leq i_1-2,\\
\theta_{i}, & i_1-1\leq i\leq i_2-1,\\
\theta_{i+1}, & i_2\leq i\leq n.
\end{cases}
\end{align*}
Assume Hypothesis \ref{hyp_1}. By using the induction hypothesis, we can write
\begin{align}\label{ine_cy''}
\cY_{t,x}^0\geq &K_1(4\pi)^{\frac{n-1}{2}}\prod_{i=2}^{n-1}(t-s_i)^{-\frac{1}{2}}.
\end{align}
Inequality \eqref{inq_lb1} follows from \eqref{ine_cy} and \eqref{ine_cy''}.

On the other hand, assume Hypothesis \ref{hyp_2}. Induction hypothesis implies that for all $(t,x)\in [(\sum_{i=1}^{n-1}(\theta_i'')^{-1})C_x ,\infty)\times \RR$,
\begin{align}\label{ine_cy''2}
\cY_{t,x}^0\geq \frac{1}{2}L(4\pi)^{\frac{n-1}{2}}t^{-\gamma}\prod_{i=2}^{n-1}(t-s_i)^{-\frac{1}{2}}.
\end{align}
Recall the construction of $\{\theta''_i:i=1,\dots, n-1\}$ and the fact that  $\theta'\in [\frac{\theta_{i_1}\theta_{i_2}}{\theta_{i_1}+\theta_{i_2}},1]$, we have
\begin{align*}
\sum_{i=1}^{n-1}(\theta_i'')^{-1}=&\sum_{\substack{1\leq i\leq n\\ i\notin \{i_1,i_2\}}}\theta_i^{-1}+(\theta')^{-1}\leq \sum_{\substack{1\leq i\leq n\\ i\notin \{i_1,i_2\}}}\theta_i^{-1}+(\theta_{i_1}^{-1}+\theta_{i_2}^{-1})=\sum_{i=1}^{n}\theta_i^{-1}.
\end{align*}
Thus inequality \eqref{ine_cy''2} holds for all $(t,x)\in[(\sum_{i=1}^n\theta_i^{-1})C_x ,\infty)\times \RR\subset [(\sum_{i=1}^{n-1}(\theta_i'')^{-1})C_x ,\infty)\times \RR$. Then, inequality \eqref{inq_lb2} is straightforward. The proof of this lemma is complete.
\end{proof}

\subsection{Completion of the proof of Theorem \ref{thm_tsbd}}\label{ss_cpl}
The proof of Theorem \ref{thm_tsbd} follows from Lemmas \ref{lmm_ub} and \ref{lmm_lb} and the next lemma of the cardinality of $\cJ_{n,n'}$.
\begin{lemma}[{\cite[Lemma 4.3]{arxiv-21-hu-kouritzin-xia-zheng}}]\label{lmm_ncj}
Let $\cJ_{n,n'}$ be defined as in \eqref{def_jnn'} with some positive integer $n$ and nonnegative integer $n'\leq n-1$. Then,
\begin{align*}
|\cJ_{n,n'}|=\frac{n!(n-1)!}{2^{n'}(n-n')!(n-n'-1)!}
\end{align*}
where by convention $0!=1$.
\end{lemma}
\begin{proof}[Proof of Theorem \ref{thm_tsbd}]
We only provide the proof of this Theorem under Hypothesis \ref{hyp_1}. It can be easily modified to cover cases under Hypothesis \ref{hyp_2}. Note that the case $n=1$ is trivial. Thus we assume that $n\geq 2$. In fact, by Theorem \ref{thm_mnt}, Lemmas \ref{lmm_ub} and \ref{lmm_lb}, we can write
\begin{align}\label{inq_tsbd11}
K_1(4\pi)^{-\frac{n}{2}}|\cJ_{n,n-1}|f_{n-1}(t)\leq \EE (u(t,x)^n) \leq \sum_{n'=0}^{n-1}(2\pi)^{-\frac{n'}{2}}K_2^{n-n'}|\cJ_{n,n'}|f_{n'}(t) 
\end{align}
where
\[
f_{n'}(t):=\int_{\mathbb{T}_{n'}^t}d\bs_{n'}\prod_{i=1}^{n'}(t-s_i)^{-\frac{1}{2}}=\frac{2^{n'}}{\Gamma(n')}t^{\frac{1}{2}n'}.
\]
Then, inequality \eqref{inq_tsbd1} follows from \eqref{inq_tsbd11}, Lemma \ref{lmm_ncj} and Stirling's formula. The proof of this theorem is complete.
\end{proof}

\section{Proofs of Propositions \ref{prop_tail} and \ref{thm_ld}}
In this section, we will provide the proofs of Propositions \ref{prop_tail} and \ref{thm_ld}.

\begin{proof}[Proof of Proposition \ref{prop_tail}]
Assume Hypothesis \ref{hyp_1}. In the first step, we prove the upper bound, namely, 
\begin{align}\label{inq_te1}
\limsup_{z\to\infty}\frac{\log\PP(|u_t(x)|>z)}{z}<-C_2t^{-\frac{1}{2}}.
\end{align}
Applying inequality \eqref{inq_tsbd1}, with $\alpha=\frac{(K^*)^{-\frac{1}{2}}t^{-\frac{1}{2}}}{2}$, we can write for any $x\in \RR$,
\begin{align*}
\EE [\exp(\alpha u_t(x))]=\sum_{n=0}^{\infty}\frac{1}{n!} \alpha^n\EE (u_t(x)^n)\leq e^{\alpha K^*}+\big(\sqrt{t}(1-\alpha K^*\sqrt{t})\big)^{-1}<\infty.
\end{align*}
Thus by using Markov's inequality, we get
\begin{align}\label{prob>z}
\PP(u_t(x)>z)\leq e^{-\alpha z} \big[e^{\alpha K^*}+\big(\sqrt{t}(1-\alpha K^*\sqrt{t})\big)^{-1}\big],
\end{align}
and thus
\begin{align*}
\frac{\log\PP(u_t(x)>z)}{z}\leq -\alpha+\frac{\xi}{z}=\frac{1}{2}(K^*)^{-\frac{1}{2}}t^{-\frac{1}{2}}+\frac{\xi}{z},
\end{align*}
where $\xi=\log [e^{\alpha K^*}+(\sqrt{t}(1-\alpha K^*\sqrt{t}))^{-1}]$ is a finite number. This proves inequality \eqref{inq_te1}. 

In the next step, we will show the next inequality,
\begin{align}\label{inq_te2}
-C_1t^{-\frac{1}{2}}\le \liminf_{z\to\infty}\frac{\log\PP(|u_t(x)|>z)}{z}.
\end{align}
By using inequality \eqref{inq_tsbd1} and the Paley-Zygmund inequality (c.f. \cite[Lemma 7.3]{ams-14-khoshnevisan}), we can write
\begin{align*}
\PP\Big(u_t(x)> \frac{1}{2} K_*&(n!)^{\frac{1}{n}}t^{\frac{n-1}{2n}}\Big)\geq
\PP \Big(u_t(x)\geq \frac{1}{2}\EE (u_t(x)^n)^{\frac{1}{n}}\Big) \geq  \bigg(1-\frac{1}{2^n}\bigg)^2 \frac{[\EE (u_t(x)^n)]^2}{\EE (u_t(x)^{2n})}\\
\geq &
 \bigg(1-\frac{1}{2^n}\bigg)^2\frac{K_*^{2n}(1+n!t^{\frac{1}{2}(n-1)})^2}{(K^*)^{2n}(1+(2n)!t^{\frac{1}{2}(2n-1)})}
\geq  \frac{9K_*^{2n}(n!t^{\frac{1}{2}(n-1)})^2}{32(K^*)^{2n}(2n)!t^{\frac{1}{2}(2n-1)}},
\end{align*}
for $n$ large enough. Taking account of Stirling's formula, we can further deduce that
\begin{align}\label{inq_plb1}
\PP\Big(u_t(x)&> \frac{1}{2}K_*(2\pi)^{\frac{1}{2n}}(n+1)^{1+\frac{1}{2n}}e^{-(1+\frac{1}{n})}t^{\frac{n-1}{2n}}\Big)\geq \PP\Big(u_t(x)\geq \frac{1}{2}K_*(n!)^{\frac{1}{n}}t^{\frac{n-1}{2n}}\Big)\geq C^nt^{-\frac{1}{2}}
\end{align}
with some universal constant $C>0$ depending on $K^*$ and $K_*$ for all $n$ large enough. Write 
\[
z=z(n) \coloneqq K_*(2\pi)^{\frac{1}{2n}}(n+1)^{1+\frac{1}{2n}}e^{-(1+\frac{1}{n})}t^{\frac{n-1}{2n}}.
\]
Notice that $\lim_{n\to \infty}(n+1)^{\frac{1}{2n}} = 1$. It follows that for $n$ large enough, $ z\leq 2K_*(2\pi)^{\frac{1}{2}}(n+1)t^{\frac{1}{2}}$,
and thus $n\geq f(t)z-1$
with $f(t)=(2K_*(2\pi)^{\frac{1}{2}}t^{\frac{1}{2}})^{-1}$.
This allows us to write that for $n$ large enough,
\begin{align*}
\PP(u_t(x)> z)\geq C^nt^{-\frac{1}{2}}\geq C^{f(t)z-1}t^{-\frac{1}{2}}.
\end{align*}
It follows that
\begin{align*}
\frac{\log \PP(u_t(x)> z)}{z}\geq \Big(f(t)-\frac{1}{z}\Big)\log C-\frac{\log t}{2z}=\Big(\frac{1}{2K_*(2\pi)^{\frac{1}{2}}t^{\frac{1}{2}}}-\frac{1}{z}\Big)\log C-\frac{\log t}{2z}.
\end{align*}
Let $n\to\infty$ (and thus $z\to\infty$), we get inequality \eqref{inq_te2}. This proves inequality \eqref{inq_te}. The proof of inequality \eqref{inq_teh2} is quite similar to that of \eqref{inq_te}, we skip it for the sake of conciseness. The proof of this theorem is complete.
\end{proof}

Note that in the proof of inequality \eqref{inq_te2} under Hypothesis \ref{hyp_1}, lower bounds for moments of sBm's of all orders are used. This prevents us to deploy the same method to get a lower bound as in \eqref{inq_te} under Hypothesis \ref{hyp_2}. However, if replacing $z$ in \eqref{inq_te} by $t^{\sigma}$ with some $\sigma \geq \frac{1}{2}$, we can deduce the next theorem, which holds under either Hypothesis \ref{hyp_1} or \ref{hyp_2} (c.f. \cite{ssp-93-iscoe-lee,ap-95-lee-remillard} for related results).

\begin{proof}[Proof of Proposition \ref{thm_ld}]
Replacing $z$ by $t^{\sigma}$ in \eqref{prob>z}, and using the fact that $e^{\alpha K^*}+(\sqrt{t}(1-\alpha K^*\sqrt{t}))^{-1}$ is decreasing in $t$, we can deduce the upper bound in \eqref{inq_ld} under Hypothesis \ref{hyp_1}. The case under Hypothesis \ref{hyp_2} is quite similar, we omit its proof for simplicity. Next, we will show the lower bound under Hypothesis \ref{hyp_1}.  For any $t$, let $n = n(t) \coloneqq  \lfloor 4 e^2 t^{\sigma - \frac{1}{2}}/K_*\rfloor$. Then, for $t$ large enough such that $t^{\frac{1}{2n}} \geq \frac{1}{2}$, we can write
\[
\frac{1}{2}K_*(2\pi)^{\frac{1}{2n}}(n+1)^{1+\frac{1}{2n}}e^{-(1+\frac{1}{n})}t^{\frac{n-1}{2n}} \geq  2 t^{\sigma-\frac{1}{2n}} \geq t^{\sigma}.
\]
 As a consequence of inequality \eqref{inq_plb1}, with a uniform constant $C>0$,
\begin{align*}
\PP (u_t(x) > t^{\sigma}) \geq \PP \Big( u_t (x) > \frac{1}{2}K_*(2\pi)^{\frac{1}{2n}}(n+1)^{1+\frac{1}{2n}}e^{-(1+\frac{1}{n})}t^{\frac{n-1}{2n}} \Big) \geq C^n t^{-\frac{1}{2}},
\end{align*}
and thus
\begin{align*}
t^{\frac{1}{2} - \sigma  }\log \PP (u_t(x) > t^{\sigma}) \geq t^{\frac{1}{2} - \sigma } \Big( \lfloor 4 e^2 t^{\sigma - \frac{1}{2}}/K_*\rfloor \log C - \frac{1}{2}\log (t)\Big).
\end{align*}
This proves the lower bound in \eqref{inq_ld} with $C_2 =  4 e^2 K_*^{-1} \log C$. On the other hand, assuming Hypothesis \eqref{hyp_2}, to apply inequality \eqref{inq_plb1}, we need  $t \geq n C_x$ (see Theorem \ref{thm_tsbd}). It holds for large $t$ as $t \sim n^{\frac{1}{\sigma - 1/2}}$ if $\frac{1}{2} < \sigma < \frac{3}{2}$. This completes the proof of Proposition \ref{thm_ld}.
\end{proof}

\begin{appendix}

\section{A moment formula for sBm's}\label{ss_mnt}

In this section, we provide a moment formula, cited from \cite[Theorem 4.1]{arxiv-21-hu-kouritzin-xia-zheng}, for sBm's.
Let $n$ be any positive integer, and let $n'<n$ be any nonnegative integer. We denote by $\cI_{n,n'}$ the collection of multi-indexes $[\alpha,\beta]=[(\alpha_1,\dots,\alpha_n),(\beta_1,\dots,\beta_{n'})]\in \{0,1\}^{n+n'}$ satisfying
\begin{enumerate}[(i)]
\item $\beta_{n'}=0$.

\item Let $|\alpha|=\sum_{i=1}^n \alpha_i$ and let $|\beta|=\sum_{i=1}^{n'}\beta_i$. Then $|\alpha|+|\beta|=2n'$.
\end{enumerate}
In particular, if $n'=0$, then $\alpha=\mathbf{0}_n$ is the $0$-vector in $\RR^n$ and $\beta$ should be a ``$0$-dimensional'' vector. In this case we write $[\alpha,\beta]=[\mathbf{0}_n,\partial]$. Fix $[\alpha,\beta]\in \cI_{n,n'}$. We introduce a map $\iota_{\alpha}:\{1,\dots, |\alpha|\}\to \{1,\dots, n\}$ by
\[
\iota_{\alpha}(i)=j_i,
\]
where the index $j_i$ is such that $\alpha_{j_i}$ is the $i$-th nonzero coordinate of $\alpha$ for all $i=1,\dots, |\alpha|$. The map $\iota_{\beta}:\{1,\dots, |\beta|\}\to \{1,\dots, n'\}$ is defined in a similar way. 

Given $[\alpha,\beta]\in \cI_{n,n'}$, let $\cK_{n,n'}^{\alpha,\beta}$ be the collection of maps $\tau:\{1,\dots, |\alpha|+|\beta|=2n'\}\to \{1,\dots, n'\}$ satisfying the following properties,
\begin{enumerate}[(i)]
\item For any $k\in \{1,\dots,n'\}$, there exist $1\leq i_1<i_2\leq 2n'$ such that $\tau(i_1)=\tau(i_2)=k$.
\item For all $i\in \{|\alpha|+1,\dots, 2n'\}$, $\tau(i)>\iota_{\beta}(i-|\alpha|)$.
\end{enumerate}
If $n'=0$, we denote $\tau=\partial$.  Finally, we write
\begin{align}\label{def_jnn'}
\cJ_{n,n'}=\{(\alpha,\beta,\tau):[\alpha,\beta]\in \cI_{n,n'},\tau\in \cK_{n,n'}^{\alpha,\beta}\}.
\end{align}
Especially, $\cJ_{n,0} = \{(\mathbf{0}_n,\partial,\partial )\}$.

\begin{theorem}\label{thm_mnt}
Suppose that $u_0\in\cM_F(\RR)$. Let $n$ be a positive integer. Then, for any $(t,x)\in\RR_+\times \RR$, the following identity  holds,
\begin{align}\label{for_mmt0}
\EE (u_{t}(x)^n)=&\sum_{n'=0}^{n-1}\sum_{(\alpha,\beta,\tau)\in \cJ_{n,n'}}\prod_{i=1}^n\Big(\int_{\RR} p_{t}(x-z)u_0(dz)\Big)^{1-\alpha_i}\nonumber\\
&\times \int_{\mathbb{T}_{n'}^t}d\bs_{n'}\int_{\RR^{n'}}d\bz_{n'}\prod_{i=1}^{n'}\Big(\int_{\RR}p_{s_i}(z_i-z)u_0(dz)\Big)^{1-\beta_i} \prod_{i=1}^{|\alpha|}p(t-s_{\tau(i)},x-z_{\tau(i)})\nonumber\\
&\times\prod_{i=|\alpha|+1}^{2n'}p(s_{\iota_{\beta}(i-|\alpha|)}-s_{\tau(i)},z_{\iota_{\beta}(i-|\alpha|)}-z_{\tau(i)}),
\end{align}
where the set $\cJ_{n,n'}$ of triples $(\alpha,\beta,\tau)$ is defined as \eqref{def_jnn'}.
\begin{align*}
\mathbb{T}_{n'}^t=\big\{\bs_{n'}=(s_1,\dots,s_{n'})\in [0,T]^{n'}:0<s_{n'}<s_{n'-1}<\dots<s_1<t\big\},
\end{align*}
and $p(t,x)=p_t(x)$ to avoid long sub-indexes.
\end{theorem}
\begin{remark}
Theorem \ref{thm_mnt} also holds for cases  when $u_0\in C_b(\RR)$. In fact, if one examines the proof of \cite[Theorem 4.1]{arxiv-21-hu-kouritzin-xia-zheng}, the initial condition does not matters, as long as the integrals appearing in the formula are all finite.
\end{remark}

\end{appendix}

\end{document}